\documentclass[11pt]{amsart}
\usepackage{amsmath}
\usepackage{amssymb}
\usepackage{amsthm}
\usepackage{amscd}
\usepackage{amsxtra}
\usepackage{verbatim}
\usepackage{color}
\usepackage{enumerate}
\usepackage{mathrsfs}

\numberwithin{equation}{section}

\definecolor{dblue}{rgb}{0,0,0.45}
\definecolor{red}{rgb}{0.7,0,0}

 \RequirePackage{geometry}
\newtheorem{theorem}{Theorem}[section]
\newtheorem{lemma}[theorem]{Lemma}
\newtheorem{corollary}[theorem]{Corollary}

\theoremstyle{definition}

\newtheorem{remark}[theorem]{Remark}

\theoremstyle{remark}

\newcommand{\N}{{\mathbb N}}
\newcommand{\R}{{\mathbb R}}

\newcommand{\f}{\frac}
\newcommand{\p}{\partial}
\newcommand{\al}{\alpha}
\newcommand{\be}{\beta}

\newcommand{\si}{\sigma}

\newcommand{\floor}[1]{\big\lfloor #1 \big\rfloor}

\begin{document}

\title{Multiplier  conditions for   Boundedness    into Hardy spaces}

\author[Grafakos]{Loukas Grafakos}
\address{Department of Mathematics, University of Missouri, Columbia, MO 65211}
\email{grafakosl@missouri.edu}

\author[Nakamura]{Shohei Nakamura}
\address{Department of Mathematical Science and Information Science}
\email{pokopoko9131@icloud.com}

\author[Nguyen]{Hanh Van Nguyen}
\address{Department of Mathematics, University of Alabama, Tuscaloosa, AL 35487}
\email{hvnguyen@ua.edu}

\author[Sawano]{Yoshihiro Sawano}
\address{Department of Mathematical Science and Information Science}
\email{yoshihiro-sawano@celery.ocn.ne.jp}

\thanks{The first author would like to thank the Simons Foundation.
The fourth author is
supported by Grant-in-Aid for Scientific Research (C), No.~16K05209, Japan Society for the Promotion of Science.}
\thanks{MSC 42B15, 42B30}

\maketitle

\begin{abstract}
In the present work, we find useful and explicit    necessary and sufficient conditions for linear and multilinear
multiplier operators of Coifman-Meyer type, finite sum of products of
Calder\'on-Zygmund operators, and also of  intermediate types to be bounded 
from a product of Lebesgue or Hardy spaces into a Hardy space. 
These conditions    state that  the symbols of the multipliers
$\sigma(\xi_1,\dots , \xi_m)$ and their derivatives vanish on  the
hyperplane $\xi_1+\cdots+\xi_m=0$.
\end{abstract}

\section{Introduction}

Hardy spaces are spaces of
distributions on $\mathbb R^n$ whose smooth maximal functions lie in $L^p(\R^n)$,
for $0<p<\infty$. These spaces coincide with $L^p({\mathbb R}^n)$ if $1<p<\infty$.
Let $0<p\le 1$ and $N$ is a prescribed integer satisfying
$N\ge \floor{n(\f 1p-1)}+1$,
where $\floor{s}$ denotes the largest integer less than or equal to $s$.
An $L^\infty$ function
$a$ is said to be $(p,\infty)$-atom, if
$a$ is supported
on some cube $Q$ and
satisfies
\begin{equation*}
\|a\|_{L^{\infty}}
\leq
1,
\quad
\int_{\R^n}
x^\alpha a(x)dx=0
\end{equation*}
for all $\alpha\in\N_0^n$ such that
$|\alpha| \leq N$, see \cite{garciaRDF}, \cite{SteinHA}. 
The space $H^p({\mathbb R}^n)$ can be characterized
as the set of all tempered distributions  
  which can be expressed as a sum of the form 
$\sum_{j=1}^\infty \lambda_j a_j$,
where  
$a_j$ are $(p,\infty)$-atoms and $(\lambda_j)_{j=1}^{\infty}$ is a sequence of non-negative numbers such that
\[
\Big\| \sum_{j=1}^{\infty}\lambda_j\chi_{Q_j} \Big\|_{L^p}<\infty.
\]

In this note we study linear or multilinear multiplier operators that map products of Hardy spaces into other Hardy spaces. These operators have the form
\begin{equation}\label{eq-TSigma}
T_{\sigma}(f_1,\ldots,f_m)(x) = \int_{\mathbb R^{mn}}e^{2\pi ix\cdot (\xi_1+\cdots+\xi_m)}\sigma(\xi_1,\ldots,\xi_m)\widehat{f_1}(\xi_1)\cdots \widehat{f_m}(\xi_m)\, d\xi_1\cdots d\xi_m,
\end{equation}
where   $\sigma$ is a bounded function on $\R^{mn}$.
Here $\widehat{f} (\xi)$ denotes the Fourier transform of a Schwartz function $f$ defined by $ \int_{\mathbb R^n} f(x) e^{-2\pi i x\cdot \xi} dx$.
We are interested in explicit   conditions on the symbol $\si$ that characterize boundedness into a Hardy space. 
These conditions reflect the amount of cancellation the symbols contain. For instance,  boundedness into $H^1({\mathbb R}^n)$  
for $m$-linear operators is characterized by the cancellation condition $\si(\xi_1,\dots, \xi_m)=0$
 on the   hyperplane
 $\Delta_n$,
 where $\Delta_n$ is given by
$$
\Delta_n = \{(\xi_1,\ldots,\xi_m)\in\mathbb{R}^{mn}\ :\ \xi_1+\cdots+\xi_m=0\}.
$$

For a multiindex $\alpha= (i_1, \dots , i_n)$ we set 
 $\partial^\al_{k} = \partial^{i_1}_{\xi_{k1}} \cdots \partial^{i_n}_{\xi_{kn}} $, where 
 $\xi_k=(\xi_{k1}, \dots , \xi_{kn}) \in \mathbb R^n$.
A symbol $\sigma(\xi_1,\dots , \xi_m)$  on $\mathbb R^{mn}$ is called of Coifman-Meyer type if
\begin{equation}\label{Coifman-Meyer}
\big|\partial^{\alpha_1}_{1} \cdots \partial^{\alpha_m}_{m} \sigma(\xi_1,\dots, \xi_m) \big| \le C_{\alpha_1,\dots, \alpha_m} (|\xi_1|+\cdots+|\xi_m|)^{-(|\alpha_1|+\cdots+|\alpha_m|)}
\end{equation}
for sufficiently large   $n$-tuples of nonnegative integers $\alpha_j$, henceforth called multiindices.   
Here $|\al|= i_1+\dots +i_n$ is the size of a multiindex
$\al=(i_1,\dots , i_n)\in \N_0^n$. 
The associated operators $T_\sigma$ are called multilinear Calder\'on-Zygmund operators; these 
were
initially introduced in \cite{CM2} and were  extensively studied in \cite{GrafakosTorresAdvances}.
These operators map products
$L^{p_1}({\mathbb R}^n)  \times \cdots \times L^{p_m}({\mathbb R}^n)$
of Lebesgue spaces
into another Lebesgue space $L^p({\mathbb R}^n)$,
where $1<p_j<\infty$, $j=1,2,\ldots,m$, 
and $0<p<\infty$ satisfy
\begin{equation}\label{indicesHolder}
\frac{1}{p}=\frac{1}{p_1}+\cdots+\frac{1}{p_m}.
\end{equation}
Boundedness into a Lebesgue space also holds if the initial spaces are Hardy spaces, as
shown in \cite{GLKT01};
the range $0<p_i<\infty$ is included in \cite{GLKT01}.
Additionally, it was shown by
the authors \cite{GNNS17}  that $T_\sigma$ maps a product of Hardy spaces 
into another Hardy space  if the action of $T_\sigma$ on atoms has vanishing moments, i.e.
\begin{equation}\label{HuMeng}
\int_{\R^{n} } x^{\al}T_\si(a_1,\dots , a_m)(x) \, dx = 0
\end{equation}
for all $(p_j,\infty)$-atom $a_j$
and
for all $|\al | \le \floor{n(\frac{1}{p}-1)}$. Remarkably, the cancellation condition 
\eqref{HuMeng} is only required to hold for all smooth functions with 
compact support $a_j\in \mathcal{O}_N(\mathbb{R}^n) $, where
\[
\mathcal{O}_N(\mathbb{R}^n)=
\bigcap_{\beta \in {\mathbb N}_0^n, |\beta| \le N}
\left\{f \in {\mathscr C}^\infty_{\rm c}({\mathbb R}^n)\,:\,
\int_{{\mathbb R}^n}x^\beta f(x)\,dx=0\right\}.
\]

We have the following theorem concerning operators associated with Coifman-Meyer symbols.

\begin{theorem}\label{Main-Thm}
 Let $\sigma$ be a bounded function on $\mathbb{R}^{mn}$ and
 $\sigma\in \mathcal C^{\infty}\big(\mathbb{R}^{mn}\setminus\{(0,\dots , 0)\}\big)$ that satisfies \eqref{Coifman-Meyer}.
 Fix $0<p_i\le \infty$, $0<p\le 1$ that satisfy \eqref{indicesHolder}.
 Then the following two statements are equivalent:
 \begin{enumerate}[(a)]
 \item $T_\sigma$ maps $H^{p_1}(\mathbb R^n)\times\cdots\times H^{p_m}(\mathbb R^n)$ to $H^p(\mathbb R^n).$
 \item For all  multiindices $\alpha$ with  $|\alpha|\le \floor{n(\frac{1}{p}-1)}$ we have 
\begin{equation}\label{MulDrivCan}
 (\partial^\al_{m} \sigma)(\xi_1,\ldots,\xi_m)  =0
\end{equation}
for all $(\xi_1,\ldots,\xi_m)\in \Delta_n \setminus\{(0,\dots , 0)\}$.
 \end{enumerate}
\end{theorem}

We also consider symbols of the product form
\begin{equation}\label{CoifGraf}
\sigma(\xi_1,\dots, \xi_m) = \sum_{j=1}^M \sigma_{j1}(\xi_1) \cdots \sigma_{jm}(\xi_m)
\end{equation}
where the $\sigma_{jk}$'s are
Fourier transforms of sufficiently smooth Calder\'on-Zygmund kernels on $\mathbb R^n$.
For such symbols with $m=2$
it was shown in  \cite{CG} (see also \cite{GLLIXINWEI})
that the associated operators are bounded
from a product of Hardy spaces into another Hardy space
if and only if \eqref{HuMeng} holds.
For symbols of the form \eqref{CoifGraf} we   prove the following analogous result:

\begin{theorem}\label{Main-Thm2}
Let $\sigma_{jk}$, $1\le j\le M,1\le k\le m$, be Fourier transforms of Calder\'on-Zygmund kernels on $\mathbb R^n$, and let
$\sigma$ be a function on $\mathbb{R}^{mn}$ given by
 \eqref{CoifGraf}.
 Fix $0<p_i<\infty$, $0<p\le 1$ that satisfy \eqref{indicesHolder}.
 Then the following two statements are equivalent:
 \begin{enumerate}[(a)]
 \item $T_\sigma$ maps $H^{p_1}(\mathbb R^n)\times\cdots\times H^{p_m}(\mathbb R^n)$ to $H^p(\mathbb R^n).$
 \item For all  multiindices $\alpha$ with  $|\alpha|\le \floor{n(\frac{1}{p}-1)}$   condition \eqref{MulDrivCan} holds, i.e.
\begin{equation*}
 (\partial^\al_{m} \sigma)(\xi_1,\ldots,\xi_m)  =0
\end{equation*}
for all $(\xi_1,\ldots,\xi_m)\in(\mathbb{R}^{n} \setminus \{0\})^m \cap \Delta_n$.
 \end{enumerate}
\end{theorem}

Note that for symbols of both types \eqref{Coifman-Meyer} and \eqref{CoifGraf} we always have
\begin{equation}\label{decay}
\big| \p_{1}^{\al_1} \cdots \p_{m}^{\al_m} \si(\xi_1,\ldots,\xi_m) \big| \le C_{\al_1,\dots , \al_m} |\xi_1|^{-|\al_1|}\cdots |\xi_m|^{-|\al_m|}
\end{equation}
for all $\alpha_j\in\N_0^n$ and all $\xi_j\in \mathbb R^n$, $j=1,\dots , m$,
under the assumption that $|\alpha_j|>0$ if $\xi_j\neq 0$.
It turns out that condition \eqref{decay} suffices for
the purposes of proving the equivalence between (a) and (b) in both Theorems \ref{Main-Thm} and \ref{Main-Thm2},
although it is not strong enough to imply  boundedness on any product of Lebesgue spaces (see \cite{GLKT02}).

\begin{remark}
By symmetry, we note that in condition \eqref{MulDrivCan} the derivative $\p_{m}^\al$ can be replaced 
by $\p_{k}^\al$ for any $k\in \{1 , \dots , m-1\}$
in Theorems \ref{Main-Thm} and \ref{Main-Thm2}.
\end{remark}

Boundedness into $H^p({\mathbb R}^n)$ for operators $T_\si$
is often expressed in terms of cancellation of the action of the operator on tuples of atoms. Let $x^\al= x_1^{i_1}
\cdots x_n^{i_n}$ if $\al=(i_1,\dots , i_n)$. 
In order for the integral
$$
\int_{\mathbb R^n}x^\alpha T_\si(a_1, \dots , a_m)(x) \, dx
$$
to be absolutely convergent, it is necessary  for $T_\si(a_1, \dots , a_m)(x)$ to have  
decay, where $a_j$ are $(p_j,\infty)$-atoms. Precisely, we assume that for any $m$-tuple of
$(p_j,\infty)$-atom $a_j$ there exists function $b\in L^{p_j}(\mathbb{R}^n)$
which decays like $|x|^{-m n-N-1}$ as $|x|\to \infty$, such that
for all $x\in \mathbb R^n$
\begin{equation}\label{EST-Tatoms}
|T_\si(a_{1}, \dots , a_{m})(x)| \lesssim
b(x).
\end{equation}

We note that condition \eqref{EST-Tatoms} is valid for a large class of multilinear operators
such as those in Theorems \ref{Main-Thm} and \ref{Main-Thm2}.
Indeed, for operators with symbols of the form \eqref{CoifGraf} we can take
\begin{equation}\label{eq:170222-1}
b(x)
=
\sum_{j=1}^{M}\prod_{k=1}^m\Big[|T_{\si_{j k}} (a_k)(x)| \chi_{Q_j^*}(x) +
\f{
|Q_{k}|^{1-\f{1}{p_k} +\f{N+1}{nm}}\chi_{(Q_k^*)^c}(x)
}
{
\big(|x-c_{k}|+\ell(Q_{k}) \big)^{n+\f{N+1}{m}}
}\Big],
\end{equation}
where $Q_k$ is a cube that contains the support of $a_k$,
$\ell(Q_k)$ denotes the length of $Q_k$.

Condition \eqref{EST-Tatoms} is also valid for
Coifman-Meyer multipliers \eqref{Coifman-Meyer}.
Indeed, we can choose
\begin{equation}\label{eq:170222-2}
b(x) = |T_{\sigma}(a_1,\ldots,a_m)(x)|\chi_{\cup_{k=1}^mQ_k^*}(x) + \prod_{k=1}^m
\f{
|Q_{k}|^{1-\f{1}{p_k} +\f{N+1}{nm}}\chi_{(Q_k^*)^c}(x)
}
{
\big(|x-c_{k}|+\ell(Q_{k}) \big)^{n+\f{N+1}{m}}
}.
\end{equation}
See \cite{GNNS17} for estimates
(\ref{eq:170222-1}) and (\ref{eq:170222-2}).

To state the main equivalence result between cancellation
of multipliers and cancellation of the
action of an operator on $m$ tuples of atoms
we introduce some notation.
For $0<\epsilon<1$ and $1\le i\le m$, we denote
\begin{equation}\label{eq-Gami}
\Gamma_{i,\epsilon}(\mathbb{R}^{mn}) = \{(\xi_1,\ldots,\xi_m)\in\mathbb{R}^{mn}\ :\ |\xi_i|\le \epsilon \},\quad
\Gamma_{\epsilon} (\mathbb{R}^{mn})= \bigcup_{i=1}^m\Gamma_{i,\epsilon}(\mathbb{R}^{mn}).
\end{equation}
We also define sets
\begin{equation}\label{eq-Gami2}
\Gamma_{i}(\mathbb{R}^{mn}) = \{(\xi_1,\ldots,\xi_m)\in\mathbb{R}^{mn}\ :\ \xi_i=0 \},\quad
\Gamma(\mathbb{R}^{mn}) = \bigcup_{i=1}^m\Gamma_{i}(\mathbb{R}^{mn}).
\end{equation}

We will derive both Theorems~\ref{Main-Thm} and ~\ref{Main-Thm2} via the following general result.

\begin{theorem}\label{Thm-MulLin}
 Let $\sigma$   in $L^\infty(\mathbb{R}^{mn} ) \cap \mathcal C^{\infty}\big(\mathbb{R}^{mn}\setminus \Gamma(\mathbb{R}^{mn})\big)$   satisfy \eqref{decay}.
 Assume that $T_\sigma $ satisfies \eqref{EST-Tatoms}
for  all $a_j\in \mathcal{O}_N(\mathbb{R}^n)$ and
$$0<p_j<\infty, \quad 1\le j\le m , \quad 0<p\le 1, \quad \frac{1}{p}=\frac{1}{p_1}+\cdots+\frac{1}{p_m}.$$
Then the following two statements are equivalent:
\begin{enumerate}
\item[(a)] For all  multiindices $\alpha$ with  $|\alpha|\le \floor{n(\frac{1}{p}-1)}$  condition \eqref{MulDrivCan} holds, i.e.
\begin{equation*}
 (\partial^\al_{m} \sigma)(\xi_1,\ldots,\xi_m)  =0
,\quad \forall\, (\xi_1,\ldots,\xi_m)\in \Delta_n\setminus\Gamma(\mathbb{R}^{mn}).
\end{equation*}
\item[(b)]
For all
$a_i\in \mathcal{O}_N(\mathbb{R}^n)$, $1\le i\le m$, condition \eqref{HuMeng} holds, i.e.
\begin{equation*}
\int_{\R^n}x^\alpha
T_{\sigma}(a_1,\ldots,a_m)(x)\, dx=0
\end{equation*}
for all $\alpha$ with
$|\alpha|\leq \floor{n(\frac{1}{p}-1)}$.
\end{enumerate}
\end{theorem}

Throughout this paper,
we denote multiindices by letters $\alpha$, $\beta$, $\gamma$, etc and use the abbreviation $\alpha \le \beta$ to denote that $\alpha_j\le \beta_j$ for all $j$ if $\al=(\al_1,\dots , \al_n)$ and $\be=(\be_1,\dots , \be_n)$.
We also let $C$ denote a   constant independent of crucial parameters  whose value may vary on different occurrences.

\section{The linear case}
In the linear case,
assumption \eqref{HuMeng} holds automatically via the following lemma:
\begin{lemma}
For any $a\in \mathcal{O}_N(\mathbb{R}^n)$ and $|\alpha|\leq N$,
we have that
\begin{equation*}
\int_{\R^n}
x^\alpha
T_\si(a)(x)dx=0.
\end{equation*}
\end{lemma}

\begin{proof}
We write
$$
\left|
\int_{\R^n}
(-2\pi ix)^\alpha
T_\si(a)(x)dx
\right|
 =
\left|
\partial^\alpha
\left[
\widehat{
T_\si(a)
}
\right]
(0)
\right|
 =
\lim_{\epsilon\to0}
\left|
\int_{\R^n}
\sigma(\xi)
\widehat{a}(\xi)
\partial^\alpha
[\varphi_{\epsilon}](\xi)
d\xi
\right|
$$
integrating by parts.
Now, we notice that
by the Taylor expansion
and
the vanishing moments of $a$,
\begin{equation*}
\widehat{a}(\xi)
=
\sum_{\beta\leq N}
C_{\beta}
\partial^{\beta}\widehat{a}(0)\xi^\beta
+
{\rm O}(|\xi|^{N+1})
=
{\rm O}(|\xi|^{N+1})
\end{equation*}
as $|\xi| \to 0$.
Hence, we see that
\begin{align*}
\left|
\int_{\R^n}
(-2\pi ix)^\alpha
T_\si(a)(x)dx
\right|
&\leq
C_{\alpha}
\lim_{\epsilon\to0}
\int_{Q(0,\epsilon)}
\left|
\sigma(\xi)
|\xi|^{|\alpha|+1}
\partial^\alpha[\varphi_{\epsilon}](\xi)]
\right|d\xi\\
&\leq
C_{\alpha}\lim_{\epsilon\to0}
\epsilon
\int_{Q(0,\epsilon)}
\left|
\sigma(\xi)
[\partial^\alpha\varphi]_{\epsilon}(\xi)
\right|d\xi\\
&\leq
C_{\alpha}\lim_{\epsilon\to0}
\epsilon
\|\sigma\|_{L^\infty}
\|\partial\varphi\|_{L^1}
=0.
\end{align*}
\end{proof}

As a result, 
the linear Fourier multipliers satisfying the suitable decay condition map
product of Hardy spaces into Hardy spaces 
as is well known.

\section{The bilinear case}

For the sake of clarity of exposition, we  first discuss the bilinear case
of Theorem~\ref{Thm-MulLin}.

\begin{theorem}\label{Thm-BiLin}
Let
$\sigma\in L^\infty(\R^n\times\R^n)
\cap
\mathcal{C}^\infty(\R^n\times\R^n\setminus \{(\xi,\eta): \,\, |\xi || \eta|=0\} )$
satisfy
$(\ref{decay})$ so that $T_\sigma$ satisfies $\eqref{EST-Tatoms}$.
Then for a given $N\in \N_0$ the following conditions are equivalent:
\begin{enumerate}
\item[(a)]
For all $\alpha \in\N^n_0$ with
$| \alpha|\leq N$
and
$\xi_1\in\R^n\setminus\{0\}$, we have
\begin{equation}\label{eq-CanDrvSigma}
\partial^\alpha_{2}
\sigma(\xi_1,-\xi_1)=0.
\end{equation}
\item[(b)]
For any smooth functions
$a_1,a_2 \in \mathcal{O}_N(\mathbb{R}^n)$,
\begin{equation}\label{eq-CanclT}
\int_{\R^n}x^\alpha
T_\si(a_1,a_2)(x)dx
=0,
\quad
\forall \quad |\alpha|\leq N.
\end{equation}
\end{enumerate}
\end{theorem}

To obtain Theorem \ref{Thm-BiLin} we need a couple of lemmas.
Here and below by $B(x,r)$ the open ball centered at $x$ of radius $r>0$.

\begin{lemma}\label{lm-axis}
Assume that $\si$ is a bounded function on $\R^{n}\times \R^{n}$ and smooth away from the axes that satisfies \eqref{decay}.
Fix $N\in \mathbb{N}_0$. Then for all $\alpha\in\N^n_0 $ with
$|\alpha|\leq N$ there is a constant $C_\al$ such that
\begin{equation}\label{extra}
\sup_{0<\epsilon<1}
\sup_{\xi_1\in\R^n\setminus B(0,2\epsilon)}
\left|
\int_{\R^n}
g(\xi-\xi_1)
\sigma(\xi_1,\xi-\xi_1)
\partial^\alpha
[
\varphi_\epsilon
](\xi)
d\xi
\right|
\leq
C_{\alpha},
\end{equation}
where $g$ is a smooth function with bounded derivatives $\partial^\beta g$ and $\partial^\beta g(0)=0$ for all $|\beta|\le N $.
\end{lemma}

\begin{proof}
Fix any
$\epsilon<1$
and any
$\xi_1\in\R^n\setminus B(0,2\epsilon)$.
We will show that
\begin{equation}\label{eq-bddInt}
\left|
\int_{\R^n}
g(\xi-\xi_1)
\sigma(\xi_1,\xi-\xi_1)
\partial^\alpha
[
\varphi_\epsilon
](\xi)
d\xi
\right|
\leq
C_{\alpha},
\end{equation}
where $C_{\alpha}$
is independent of $\epsilon$
and $\xi_1$.
Note that
the function
$\xi \mapsto \sigma(\xi_1,\xi-\xi_1)$
is smooth on
the domain of integration
$|\xi|< \epsilon $,
since $\xi_1\notin B(0,2\epsilon)$
and thus
$| \xi-\xi_1| \ge \epsilon$.
With this in mind, 
involving the Taylor expansion of $g$,
we notice that
\begin{align*}
&\left|
\int_{\R^n}
g(\xi-\xi_1)
\sigma(\xi_1,\xi-\xi_1)
\partial^\alpha
[
\varphi_\epsilon
](\xi)
d\xi
\right|\\
\leq&
C \sum_{\beta\leq\alpha}
\binom{\alpha}{\beta}
\left|
\int_{\R^n}
\partial^{\beta}g(\xi-\xi_1)
\partial_2^{\alpha-\beta}\sigma(\xi_1,\xi-\xi_1)
\varphi_\epsilon
(\xi)
d\xi
\right|\\
\leq&
C_{\alpha}  \|\varphi\|_{L^1}
\max_{ \be \le \al }
\;\sup_{\xi \in {\mathbb R}^n \setminus B(\xi_1,\epsilon)} \big|\partial^{\beta}g(\xi-\xi_1)
\partial_2^{\alpha-\beta}\sigma(\xi_1,\xi-\xi_1) \big|
\\
\leq&
C_{\alpha,\sigma}' \|\varphi\|_{L^1}
\big[\max_{ \be \le \al }\sup_{\xi \in \mathbb{R}^n\setminus \{\xi_1\}}|\partial^{\beta}g(\xi-\xi_1)|
 |\xi-\xi_1|^{|\be|-|\al|} \big] 
 =:C^{''}_{\alpha,\sigma,g,\varphi}<\infty,
\end{align*}
for any $\alpha\in\N^n_0$
with $|\alpha|\leq N$.
Here we used assumption (\ref{decay})
and the fact that
 $\partial^{\beta}g$ are bounded and vanishing at $0$ for all $|\beta|\le N $.
\end{proof}

\begin{lemma}\label{lem-NoLim}
Given $a_1,a_2\in \mathcal{O}_N(\mathbb{R}^n)$ and
$\sigma$ in $ L^\infty(\R^n\times\R^n)
\cap \mathcal{C}^\infty(\R^n\times\R^n\setminus \{(\xi,\eta): \,\, |\xi || \eta|=0\} )$
that satisfies $(\ref{decay})$, if $T_\si(a_1,a_2)$ has sufficient decay \eqref{EST-Tatoms}, then we have
 \begin{equation}\label{eq-NoLim}
 \int_{\R^n} (-2\pi ix)^\alpha T_\si(a_1,a_2)(x)dx
= \sum_{\beta\le \alpha} \binom{ \al }{ \be }
 \int_{\mathbb{R}^n}\widehat{a_1}(\xi_1)
 \partial^{\alpha-\beta} \widehat{a_2}(-\xi_1)\partial^{\beta}_2\sigma(\xi_1,-\xi_1)\; d\xi_1.
 \end{equation}
\end{lemma}

\begin{proof}
First, we write
\begin{align}\label{xTa-LimT}
\int_{\R^n}
(-2\pi i x)^\al
T_\si(a_1,a_2)(x)
dx
& =
\partial^\alpha
 \left[
\widehat{T_\si(a_1,a_2)}
 \right] (0) \notag \\
& =
\lim_{\epsilon\to0}(-1)^{|\alpha|}
\int_{\R^n}
\widehat{T_\si(a_1,a_2)}(\xi)
\partial^\alpha
[
\varphi_\epsilon
](\xi)
d\xi
\end{align}
using integration by parts.
In view of the identity
\begin{equation}\label{formula}
\widehat{T_{\sigma}(a_1,a_2)}(\xi)
=
\int_{\R^n}
\widehat{a_1}(\xi_1)
\widehat{a_2}(\xi-\xi_1)
\sigma(\xi_1,\xi-\xi_1)
d\xi_1,
\end{equation}
the expression on the right in
(\ref{xTa-LimT}) equals
\begin{equation}\label{LimEps}
\lim_{\epsilon\to0}
(-1)^{|\alpha|}
\int_{\R^n}
\widehat{a_1}(\xi_1)
\left(
\int_{\R^n}
\widehat{a_2}(\xi-\xi_1)\sigma(\xi_1,\xi-\xi_1)
\partial^\alpha
[
\varphi_\epsilon
](\xi)
d\xi
\right)
d\xi_1.
\end{equation}
Now, we decompose
(\ref{LimEps}) as
$
\lim_{\epsilon\to0}
({\rm I}_\epsilon
+
{\rm II}_\epsilon)
$,
where
\begin{align*}
&{\rm I}_{\epsilon}
:=
(-1)^{|\alpha|}
\int_{B(0,2\epsilon)}
\widehat{a_1}(\xi_1)
\left(
\int_{\R^n}
\widehat{a_2}(\xi-\xi_1)\sigma(\xi_1,\xi-\xi_1)
\partial^\alpha
[
\varphi_\epsilon
](\xi)
d\xi
\right)
d\xi_1,\\
&{\rm II}_{\epsilon}
:=
(-1)^{|\alpha|}
\int_{\R^n\setminus B(0,2\epsilon)}
\widehat{a_1}(\xi_1)
\left(
\int_{\R^n}
\widehat{a_2}(\xi-\xi_1)
\sigma(\xi_1,\xi-\xi_1)
\partial^\alpha
[\varphi_\epsilon](\xi)
d\xi
\right)
d\xi_1.
\end{align*}
For the first term,
 using the vanishing moment condition
for $a_1$,
we have that
\begin{equation*}
|
{\rm I}_{\epsilon}
|
\leq
C
\|\widehat{a_2}\|_{L^\infty}
\|\sigma\|_{L^\infty}
\|\partial^\alpha\varphi\|_{L^1}
\int_{B(0,2\epsilon)}
|\xi_1|^{N}
\epsilon^{-|\alpha|}
d\xi_1
\leq
C\epsilon^{N-|\al|+n}
\to0
\quad
(\epsilon\to0).
\end{equation*}
For the second term,
inequality \eqref{extra}
gives us
\begin{equation}\label{160226-6}
\left|
\int_{\R^n}
\widehat{a_2}(\xi-\xi_1)
\sigma(\xi_1,\xi-\xi_1)
\partial^\alpha
[\varphi_\epsilon](\xi)
d\xi
\right|
\leq
C_{\alpha},
\end{equation}
for any $\epsilon\in(0,1)$
and any
$\xi_1\in\R^n\setminus B(0,2\epsilon)$
where the constant $C_\alpha$
is independent of $\epsilon$
and $\xi_1$. Recall $\partial_2$ the derivative with respect to the second variable of a function of two variables. 
Integrating by parts, we rewrite $\mathrm{II}_{\epsilon}$ as
\[
\mathrm{II}_{\epsilon} 
= 
(-1)^{|\alpha|}
\int_{\R^n\setminus B(0,2\epsilon)}
\widehat{a_1}(\xi_1)
\left(
\sum_{\beta\leq\alpha} \binom{\al}{\be}
\int_{\R^n}
\partial^{\alpha-\beta}{\widehat{a_2}}(\xi-\xi_1)
\partial_2^{\beta}\sigma(\xi_1,\xi-\xi_1)
\varphi_\epsilon
(\xi)
d\xi
\right)
d\xi_1.
\]
The Lebesgue dominated convergence theorem
and
the approximation to identity, combined with the fact
that \eqref{160226-6} holds and that $\widehat{a_1}\in L^1(\mathbb R^n)$, yields
\begin{equation*}
\lim_{\epsilon\to0}
{\rm II}_\epsilon\\
=
\sum_{\beta\le \alpha}\binom{\al}{\be}
 \int_{\mathbb{R}^n}\widehat{a_1}(\xi_1)
 \partial^{\alpha-\beta} \widehat{a_2}(-\xi_1)\partial^{\beta}_2\sigma(\xi_1,-\xi_1)\; d\xi_1.
\end{equation*}
This completes the proof of the lemma.
\end{proof}

\begin{lemma}\label{FuncZeta}
 There exists a function
$\zeta\in \mathcal C^{\infty}_0(\mathbb{R}^n)$
such that
\begin{equation}\label{eq:160331-1}
\{\xi \in B(0,1)\,:\,\widehat{\zeta}(\xi)= 0\}=\{0\}.
\end{equation}
\end{lemma}

\begin{proof}
The Fourier transform of the function
$\big( \f{\cos |\xi| - 1}{|\xi|}\big)^{n+1}$ on $\mathbb{R}^n$ is
known to be compactly supported;
see \cite[Lemma 3.1]{BGSY} and bounded but may not be smooth.
Let $\Phi$ be a smooth and compactly supported function with non-vanishing integral. Then
$\zeta = \Phi * \Big(\big( \f{\cos |\xi| - 1}{|\xi|}\big)^{n+1}\Big)\spcheck $ lies in $\mathcal C^{\infty}_0(\mathbb{R}^n)$ and satisfies
$\widehat{\zeta}(\xi)\ne 0$ for all $0\ne\xi$ in a neighborhood of the origin, since $\widehat{\Phi}$ and $\cos |\xi| - 1$ do not vanish near zero and $\cos |\xi| - 1$ vanishes only at zero.
It remains to dilate $\zeta$ to make it satisfy
(\ref{eq:160331-1}).
\end{proof}

\begin{lemma} \label{lm-Sawa} 
Let $N\in\N$ be fixed and $F \in L^\infty({\mathbb R}^n)$.
Assume for all functions $G \in L^\infty_{\rm c}({\mathbb R}^n)$
with $\widehat{G} \in L^1({\mathbb R}^n)$ satisfying
\[
\int_{{\mathbb R}^n}x^\alpha G(x)\,dx=0 \qquad \forall \,\, |\al| \le N,
\]
we have
\[
\int_{{\mathbb R}^n}\widehat{G}(\xi)F(\xi)\,d\xi=0 ,
\]
Then
$F=0$ a.e..
\end{lemma}

\begin{proof}
Denote
\[
\Omega_N(\mathbb{R}^n) = \big\{f\in L^{\infty}_{\rm c}(\mathbb{R}^n)\ :\ \widehat{f}\in L^1(\mathbb{R}^n),\int_{\R^n} x^{\alpha}f(x)\; dx=0,\quad\forall |\alpha|\le N\big\}.
\]
First, we observe that if $G\in\Omega_N(\mathbb{R}^n)$, then $G_{x_0}\in\Omega_N(\mathbb{R}^n),$ where $G_{x_0} = G(\cdot-x_0)$ 
for given $x_0\in \mathbb{R}^n$. To check this observation for $G\in\Omega_N(\mathbb{R}^n)$, we can easily see that $G_{x_0}$ is a bounded function with bounded support. Also $\widehat{G_{x_0}}(\xi) = e^{2\pi ix_0\cdot\xi}\widehat{G}(\xi)$; 
and hence $\widehat{G_{x_0}}\in L^1({\mathbb R}^n),$ since $\widehat{G}\in L^1({\mathbb R}^n).$
Next we want to show that
\begin{equation}\label{Gx0Can}
 \int_{\R^n} x^{\alpha}G_{x_0}(x)\; dx =0,\qquad\forall|\alpha|\le N.
\end{equation}
In fact, we have
\begin{align*}
 \int_{\R^n} x^{\alpha}G_{x_0}(x)\; dx
&= \int_{\R^n} (x+x_0)^{\alpha}G(x)\; dx \\
 &= \sum_{\beta\le\alpha}C_{\alpha,\beta}(x_0)
\int_{\R^n} x^{\beta}G(x)\; dx=0,\quad\forall |\alpha|\le N.
\end{align*}
Thus \eqref{Gx0Can} is verified, and we are done with checking that $G_{x_0}\in \Omega_N(\mathbb{R}^n)$.

As a consequence of the above observation, we claim that $\widehat{G}F=0$ a.e. and for all $G\in\Omega_N(\mathbb{R}^n)$. Indeed, fix $G\in\Omega_N(\mathbb{R}^n).$ For each $x_0\in\mathbb{R}^n,$ the above observation showed that $G_{x_0} = G(\cdot-x_0)\in \Omega_N(\mathbb{R}^n).$ Therefore,
\[
\int_{\R^n} \widehat{G}(\xi)F(\xi)e^{2\pi ix_0\cdot\xi}\;d\xi
=
\int_{\R^n} \widehat{G_{x_0}}(\xi)F(\xi)\;d\xi=0,
\]
i.e., $\big(\widehat{G}F)\spcheck(x_0) = 0$ for each $x_0\in \mathbb{R}^n,$ and for all $G\in\Omega_N(\mathbb{R}^n)$. This completes our claim $\widehat{G}F=0$ a.e. and for all $G\in\Omega_N(\mathbb{R}^n)$.

The rest of the proof is to verify that $F=0$ a.e.
by showing $F=0$ a.e. on $B(0,1)$.
By Lemma \ref{FuncZeta},
we can find a function $\zeta\in C^{\infty}_0(\mathbb{R}^n)$
such that $\widehat{\zeta}(0)=0$ and $\widehat{\zeta}(\xi)\ne 0$
for all $0<|\xi|<1.$ Define
\[
G=  \underbrace{\zeta * \cdots * \zeta}_{N+1 \,\, \textup{times}}.
\]
It is clear that $G  \in C^{\infty}_0(\mathbb{R}^n)$ and
\[
\widehat{G}(\xi) = \big[\zeta(\xi)\big]^{N+1},
\]
which satisfies condition $\partial^{\alpha}\widehat{G}(0)=0$ for all $|\alpha|\le N.$
Thus $G \in\Omega_N(\mathbb{R}^n)$. By our claim, we have $\widehat{G}F=0$ a.e. Noting that $\widehat{G}(\xi)\ne 0$
for $0<|\xi|<1,$
we deduce $F=0$ a.e. on $B(0,1)$.
By a suitable dilation, we can show that $F=0$ a.e. on $\mathbb{R}^n$.
\end{proof}

\begin{proof}[Proof of Theorem \ref{Thm-BiLin}]
We first assume \eqref{eq-CanDrvSigma}, and then prove \eqref{eq-CanclT}.
This direction can be obtained easily by Lemma \ref{lem-NoLim}.

Next we consider the inverse implication, i.e.,
assume \eqref{eq-CanclT} and then prove \eqref{eq-CanDrvSigma}.
We first focus on the case of
$\alpha=0$.
By Lemma \ref{lem-NoLim}, condition \eqref{eq-CanclT} is equivalent to
\[
\int_{\R^n} \widehat{a_1}(\xi_1)\widehat{a_2}(-\xi_1)
\sigma(\xi_1,-\xi_1)\; d\xi_1=0
\]
for all $H^{p_1}$-atoms $a_1$ and for all $H^{p_2}$-atoms $a_2$.
Now Lemma \ref{lm-Sawa} implies that
\begin{equation}\label{AtomSgmVanish}
\widehat{a_2}(-\xi_1)\sigma(\xi_1,-\xi_1)=0,\quad\forall \xi_1\ne 0.
\end{equation}
Fix $\xi_1\in\R^n,\xi_1\ne 0.$
Choose $a_2\in C^{\infty}_0(\mathbb{R}^n)$,
such that
$\widehat{a_2}(-\xi_1)>0$, and hence \eqref{AtomSgmVanish} deduces
$\sigma(\xi_1,-\xi_1)=0$,
which implies \eqref{eq-CanDrvSigma} for $\alpha=0$.

Next, we discuss the case of
$|\alpha|\ge 1$ by induction on its order.
Indeed, assume inductively
that \eqref{eq-CanDrvSigma} holds for all $|\alpha|\le k<N.$
We want to show that it also holds for $|\alpha|= k+1\le N$.
The inductive hypothesis together with Lemma \ref{lem-NoLim} deduces
\[
\int_{\R^n} \widehat{a_1}(\xi_1)\widehat{a_2}(-\xi_1)
\partial^{\alpha}_2\sigma(\xi_1,-\xi_1)\; d\xi_1=0.
\]
Repeat the argument in the case $\alpha=0$, we obtain \eqref{eq-CanDrvSigma} for $|\alpha|=k+1.$ The proof of the theorem is now completed.
\end{proof}

\section{The multilinear case}

In this section we prove Theorem~\ref{Thm-MulLin}.

\begin{lemma}\label{lem-NoLimMul}
Let $N \in {\mathbb N}$.
Let $\alpha$ be a multi-index with $|\alpha| \le N$.
Let $\sigma$ and $a_i$ be functions
as stated in Theorem \ref{Thm-MulLin}.
Then we have
 \begin{multline}
 \label{eq-NoLimMul}
 \int_{\R^n} (-2\pi ix)^\alpha T_{\sigma}(a_1,\ldots,a_m)(x)\; dx\\ = \sum_{\beta\le \alpha}
 \binom{\alpha}{\beta}
 \int_{\mathbb{R}^{(m-1)n}}\widehat{a_1}(\xi_1)\cdots \widehat{a_{m-1}}(\xi_{m-1})
 \partial^{\alpha-\beta} \widehat{a_m}(-\xi_1-\cdots-\xi_{m-1})\times\\
 \times\partial^{\beta}_m\sigma(\xi_1,\ldots,\xi_{m-1},-\xi_1-\cdots-\xi_{m-1})\;
 d\xi_1\cdots d\xi_{m-1}.
 \end{multline}
\end{lemma}
\begin{proof}
 Recall the function $\varphi$ supported in the unit ball and $\widehat{\varphi}(0)=1.$
Fix $a_j\in \mathcal{O}(\mathbb{R}^n)$, $1\le j\le m.$ Now we have
 \begin{multline}\label{xTa-LimTMul}
\int_{\R^n} (-2\pi ix)^\alpha T_\sigma(a_1,\ldots,a_m)(x)\; dx\\
= \partial^\alpha
\left[
\widehat{T_\sigma(a_1,\ldots,a_m)}
\right](0)
=
\lim_{\epsilon\to0}
\int_{\R^n}
\widehat{T_\sigma(a_1,\ldots,a_m)}(\xi)
\partial^\alpha
[
\varphi_\epsilon
](\xi)
d\xi\\
= \lim_{\epsilon\to0}
\int_{\mathbb{R}^{mn}}
\widehat{a_1}(\xi_1)\cdots\widehat{a_{m}}(\xi_{m})
\sigma(\xi_1,\ldots,\xi_{m})
\partial^\alpha[
\varphi_\epsilon
](\xi_1+\cdots+\xi_m)
d\xi_1\cdots d\xi_{m}.
\end{multline}
Let
\[
\Delta^{m-1}_\epsilon = \{(\xi_1,\ldots,\xi_m)\in\mathbb{R}^{mn}\ :\ |\xi_1+\cdots+\xi_{m-1}|\le 2\epsilon\},
\]
and denote
\[
\Sigma^0_\epsilon = \big(\cup_{i=1}^{m-1}\Gamma_{i,\epsilon}(\mathbb{R}^{mn})\big)\cup \Delta^{m-1}_\epsilon,
\]
where $\Gamma_{i,\epsilon}(\mathbb{R}^{mn})$ is defined in \eqref{eq-Gami}. Also set $\Sigma^1_\epsilon = \mathbb{R}^{mn}\setminus \Sigma^0_\epsilon$, and hence $\mathbb{R}^{mn} = \Sigma^0_\epsilon\cup \Sigma^1_\epsilon$. The last integral in \eqref{xTa-LimTMul} can be decomposed into two parts: $\mathrm{I}_\epsilon+\mathrm{II}_\epsilon,$ where
\[
\mathrm{I}_\epsilon =
\int_{\Sigma^0_\epsilon}
\widehat{a_1}(\xi_1)\cdots\widehat{a_{m}}(\xi_{m})
\sigma(\xi_1,\ldots,\xi_{m})
\partial^\alpha[
\varphi_\epsilon
](\xi_1+\cdots+\xi_m)
d\xi_1\cdots d\xi_{m}
\]
and
\[
\mathrm{II}_\epsilon =
\int_{\Sigma^1_\epsilon}
\widehat{a_1}(\xi_1)\cdots\widehat{a_{m}}(\xi_{m})
\sigma(\xi_1,\ldots,\xi_{m})
\partial^\alpha[
\varphi_\epsilon
](\xi_1+\cdots+\xi_m)
d\xi_1\cdots d\xi_{m}.
\]
Next we will show that $\lim_{\epsilon\to0}\mathrm{I}_\epsilon = 0.$ Indeed, we can estimate
\begin{align*}
 |\mathrm{I}_\epsilon| \le& \sum_{i=1}^{m-1}
 \big|\int_{\Gamma_{i,\epsilon}(\mathbb{R}^{mn})}
\widehat{a_1}(\xi_1)\cdots\widehat{a_{m}}(\xi_{m})
\sigma(\xi_1,\ldots,\xi_{m})
\partial^\alpha[
\varphi_\epsilon
](\xi_1+\cdots+\xi_m)
d\xi_1\cdots d\xi_{m}\big|
 \\
 &+ \big|\int_{\Delta^{m-1}_\epsilon}
\widehat{a_1}(\xi_1)\cdots\widehat{a_{m}}(\xi_{m})
\sigma(\xi_1,\ldots,\xi_{m})
\partial^\alpha[
\varphi_\epsilon
](\xi_1+\cdots+\xi_m)
d\xi_1\cdots d\xi_{m}\big|.
\end{align*}
Thus, it is enough to show that
\begin{equation}\label{eq-partI}
 \lim_{\epsilon\to0}
\int_{\Gamma_{i,\epsilon}(\mathbb{R}^{mn})}
\widehat{a_1}(\xi_1)\cdots\widehat{a_{m}}(\xi_{m})
\sigma(\xi_1,\ldots,\xi_{m})
\partial^\alpha[
\varphi_\epsilon
](\xi_1+\cdots+\xi_m)
d\xi_1\cdots d\xi_{m}=0,
\end{equation}
for all $1\le i\le m-1$,
and
\begin{equation}\label{eq-partII}
 \lim_{\epsilon\to0}
 \int_{\Delta^{m-1}_\epsilon}
\widehat{a_1}(\xi_1)\cdots\widehat{a_{m}}(\xi_{m})
\sigma(\xi_1,\ldots,\xi_{m})
\partial^\alpha[
\varphi_\epsilon
](\xi_1+\cdots+\xi_m)
d\xi_1\cdots d\xi_{m}=0.
\end{equation}
Without loss of generality, we have only to prove \eqref{eq-partI} for $i=1.$ In this case, we have
\[
|\widehat{a_1}(\xi)| \le C(a_1)\min(1,|\xi|^{N+1})
\le C(a_1)|\xi|^{|\al|+1}
\]
and hence
\begin{align*}
 \Big|\int_{\Gamma_{1,\epsilon}(\mathbb{R}^{mn})}&
\widehat{a_1}(\xi_1)\cdots  \, \widehat{a_{m}}(\xi_{m})
\sigma(\xi_1,\ldots,\xi_{m})
\partial^\alpha[
\varphi_\epsilon
](\xi_1+\cdots+\xi_m)
d\xi_1\cdots d\xi_{m}\Big| \\
&\le
\int_{\Gamma_{1,\epsilon}(\mathbb{R}^{mn})} \Big|
\widehat{a_1}(\xi_1)\cdots, \widehat{a_{m}}(\xi_{m})
\sigma(\xi_1,\ldots,\xi_{m})
\partial^\alpha[
\varphi_\epsilon
](\xi_1+\cdots+\xi_m)\Big|
d\xi_1\cdots d\xi_{m} \\
&\,\le   C(a_1)\|\partial^\alpha\varphi\|_{L^\infty}
\|\widehat{a_2}\|_{L^1}\cdots \|\widehat{a_{m-1}}\|_{L^1}
\|\widehat{a_m}\|_{L^1}
\|\sigma\|_{L^\infty}
\epsilon^{-|\alpha|-n}
\int_{B(0,2\epsilon)}|\xi_1|^{|\alpha|+1}\; d\xi_1\\
&\,\le  C(a_1)\|\partial^\alpha\varphi\|_{L^\infty}
\|\widehat{a_2}\|_{L^1}\cdots \|\widehat{a_{m-1}}\|_{L^1}
\|\widehat{a_m}\|_{L^1}
\|\sigma\|_{L^\infty}\epsilon,
\end{align*}
which tends to $0$ as $\epsilon$ approaches to $0$. 

Notice that $\varphi$ is supported in the unit ball, therefore $\varphi_{\epsilon}(\xi_1+\cdots+\xi_m)$ survives 
only if $|\xi_1+\cdots+\xi_m| \le \varepsilon$.
Identity \eqref{eq-partII} can be proved similarly by making use of the fact that for all $(\xi_1,\ldots,\xi_m)\in \Delta^{m-1}_\epsilon$,
\[
|\xi_m| \le |\xi_1+\cdots+\xi_m|+|\xi_1+\cdots+\xi_{m-1}|\le 3\epsilon,
\]
and the vanishing moments of $a_m$.

Now we turn into $\mathrm{II}_\epsilon$ and rewrite it in the following form
\begin{multline*}
 \mathrm{II}_\epsilon = \int_{\substack{|\xi_1|>\epsilon,\ldots,|\xi_{m-1}|>\epsilon\\
 |\xi_1+\cdots+\xi_{m-1}|>2\epsilon}}
 \widehat{a_1}(\xi_1)\cdots \widehat{a_{m-1}}(\xi_{m-1})\int_{\mathbb{R}^n}\widehat{a_m}(\xi-\xi_1-\cdots-\xi_{m-1})\times\\
 \times\sigma(\xi_1,\ldots,\xi_{m-1},\xi-\xi_1-\cdots-\xi_{m-1})
 \partial^\alpha\big[\varphi_{\epsilon}\big](\xi)\;d\xi\;d\xi_1\cdots d\xi_{m-1}.
\end{multline*}
Fix $\xi_1,\ldots,\xi_{m-1}$ so that $|\xi_1+\cdots+\xi_{m-1}|>2\epsilon$,
and that $|\xi_i|>\epsilon$ for all $1\le i\le m-1.$
We easily see that the function
$\xi\mapsto \sigma(\xi_1,\ldots,\xi_{m-1},\xi-\xi_1-\cdots-\xi_{m-1})$
is smooth on $B(0,\epsilon)$.
Integrating by parts, we have
\begin{multline*}
 \int_{\mathbb{R}^n}\widehat{a_m}(\xi-\xi_1-\cdots-\xi_{m-1})
 \sigma(\xi_1,\ldots,\xi_{m-1},\xi-\xi_1-\cdots-\xi_{m-1})
 \partial^\alpha\big[\varphi_{\epsilon}\big](\xi)\;d\xi \\
 =\sum_{\beta\le\alpha}
 \binom{\alpha}{\beta}
 \int_{\mathbb{R}^n}\partial^{\alpha-\beta}\widehat{a_m}(\xi-\xi_1-\cdots-\xi_{m-1})
 \partial^\beta_{m}\sigma(\xi_1,\ldots,\xi_{m-1},\xi-\xi_1-\cdots-\xi_{m-1})
 \varphi_{\epsilon}(\xi)\;  d\xi.
\end{multline*}
Thus
\begin{multline*}
 \mathrm{II}_\epsilon = \sum_{\beta\le\alpha}
 \binom{\alpha}{\beta}
 \int_{\substack{|\xi_1|>\epsilon,\ldots,|\xi_{m-1}|>\epsilon\\
 |\xi_1+\cdots+\xi_{m-1}|>2\epsilon}}
 \widehat{a_1}(\xi_1)\cdots \widehat{a_{m-1}}(\xi_{m-1})\Big\{
 \int_{\mathbb{R}^n}\partial^{\alpha-\beta}\widehat{a_m}(\xi-\xi_1-\cdots-\xi_{m-1})\\
 \partial^\beta_{m}\sigma(\xi_1,\ldots,\xi_{m-1},\xi-\xi_1-\cdots-\xi_{m-1})
 \varphi_{\epsilon}(\xi)\;d\xi\Big\}
 \;d\xi_1\cdots d\xi_{m-1}.
\end{multline*}
An argument similar to Lemma \ref{lm-axis} allows us
to use Lebesgue dominated convergence theorem
to pass the limit to inside the above integral, together with the use of the approximate   identity, to obtain
\begin{multline*}
 \lim_{\epsilon\to 0}\mathrm{II}_\epsilon = \sum_{\beta\le\alpha}
 \binom{\alpha}{\beta}
 \int_{\substack{|\xi_1|>0,\ldots,|\xi_{m-1}|>0\\
 |\xi_1+\cdots+\xi_{m-1}|>0}}
 \widehat{a_1}(\xi_1)\cdots \widehat{a_{m-1}}(\xi_{m-1})\partial^{\alpha-\beta}\widehat{a_m}(-\xi_1-\cdots-\xi_{m-1})\\
 \partial^\beta_{m}\sigma(\xi_1,\ldots,\xi_{m-1},-\xi_1-\cdots-\xi_{m-1})
 \;d\xi_1\cdots d\xi_{m-1}.
\end{multline*}
This identity completes the proof of the lemma.
\end{proof}
\begin{proof}[Proof of Theorem \ref{Thm-MulLin}]
 By Lemma \ref{lem-NoLim}, it is clear that if \eqref{MulDrivCan} is valid then \eqref{HuMeng} holds automatically. For the reverse direction, we use an analogous extension of Lemma \ref{lm-Sawa} and
 repeat the proof of Theorem \ref{Thm-BiLin}.
\end{proof}

\section{Proof of Theorem \ref{Main-Thm}}
Let $N \in {\mathbb N}$ be fixed
and let $\sigma$ be a bounded function in $\mathbb{R}^n$ that satisfies either \eqref{Coifman-Meyer} 
or \eqref{CoifGraf}, and let $T_\sigma$ be the multilinear multiplier operator associated to $\sigma$. As showed in \cite{GNNS17}, $T_{\sigma}$ is bounded 
from $H^{p_1}({\mathbb R}^n)\times\cdots\times H^{p_m}({\mathbb R}^n)$ to $H^p({\mathbb R}^n)$, 
where $0<p\le 1, 0<p_j<\infty$ and $\frac1{p}=\frac1{p_1}+\cdots+\frac1{p_m}$, if \eqref{HuMeng} holds, i.e.,
\[
\int_{{\mathbb R}^n} x^{\alpha}T_{\sigma}(a_1,\ldots,a_m)(x)dx=0,
\]
for all $a_j\in \mathcal{O}_N(\mathbb{R}^n)$ and all $0<|\alpha|\le \floor{n(\frac1p-1)}$. Therefore, 
the reverse direction from $(b)$ to $(a)$ of Theorem \ref{Main-Thm} follows from Theorem \ref{Thm-MulLin}.

To obtain the other direction, since $T_\sigma$ satisfies \eqref{EST-Tatoms}, 
$|x|^N T_\sigma(a_1,\ldots,a_m)$ is an integrable function. 
Therefore if $T_\sigma(a_1,\ldots,a_m)\in H^p({\mathbb R}^n)$, 
then \eqref{HuMeng} is valid. This is a consequence of a result in \cite[p. 128, 5.4 (c)]{SteinHA}. 
Similarly, we can prove Theorem \ref{Main-Thm2} by repeating the above argument.

\section{Remarks, Examples, and Applications}

It is noteworthy to mention that our results are also valid for
 symbols of intermediate or {\it mixed type}, i.e., of the form
\begin{equation}\label{eq.CalZygOPT-3}
 \sigma(\xi_1,\ldots,\xi_m) 
=
\sum_{\rho=1}^T
\sum_{\substack{I_1^\rho,\ldots,I_{G(\rho)}^\rho }}
\prod_{g=1}^{G(\rho)}
 \sigma_{I_g^\rho}   (\{\xi_l\}_{l \in I_g^{ \rho } }) ,
\end{equation}
where for each $\rho=1,\ldots,T$, $I^\rho_1, \ldots ,I^\rho_{G(\rho)}$ is a partition of 
$\{1,\ldots,m\}$ 
and each $T_{\sigma_{I_g^\rho}}$ is an $| I_g^\rho|$-linear Coifman-Meyer multiplier operator.
We write $I^\rho_1+ \cdots + I^\rho_{G(\rho)}=\{1,\ldots,m\}$ to denote such partitions. 
There is an analogous theorem for these general symbols.

\begin{theorem}\label{Main-Thm3}
Let $\sigma $ be as in \eqref{eq.CalZygOPT-3}. 
 Fix $0<p_i<\infty$, $0<p\le 1$ that satisfy \eqref{indicesHolder}.
 Then the following two statements are equivalent:
 \begin{enumerate}[(a)]
 \item $T_\sigma$ maps $H^{p_1}(\mathbb R^n)\times\cdots\times H^{p_m}(\mathbb R^n)$ to $H^p(\mathbb R^n).$
 \item For all $|\alpha|\le \floor{n(\frac{1}{p}-1)}$ condition \eqref{MulDrivCan} holds, i.e.
\begin{equation*}
\partial^\alpha_{m}
\sigma(\xi_1,\ldots,\xi_m)=0
\end{equation*}
for all $(\xi_1,\ldots,\xi_m)$ on the hyperplane $\Delta_n$ away from the points of singularity of $\sigma$.
 \end{enumerate}
\end{theorem}

For the sake of brevity we don't include a proof of Theorem~\ref{Main-Thm3} in this note, but we point out that similar techniques can be used to obtain it. 

Next, we provide examples of functions that satisfy conditions \eqref{eq-CanDrvSigma}; some of these examples are inspired by those given in \cite{G}: On $\mathbb R^2\times \mathbb R^2$ with coordinates $(\xi_1,\eta_2,\eta_1,\eta_2)$
consider the multipliers
\begin{align*}
\si_0(\xi_1,\xi_2,\eta_1,\eta_2)
&=
 \f{ \xi_1\eta_2 - \xi_2\eta_1 }{
{ |\xi_1|^2+ |\xi_2|^2+ |\eta_1|^2+|\eta_2|^2 } }\\
&=
 \f{1}{
{ |\xi_1|^2+ |\xi_2|^2+ |\eta_1|^2+|\eta_2|^2 } }
\det
\begin{pmatrix}
\xi_1&\xi_2\\
\eta_1&\eta_2
\end{pmatrix} \, .
\end{align*}
An alternative example is obtained by considering the multiplier
\begin{align*}
\si_1(\xi_1,\xi_2,\eta_1,\eta_2)
&=
 \f{ \xi_1\eta_2 - \xi_2\eta_1 }{\sqrt{ |\xi_1|^2+ |\xi_2|^2 } \, \sqrt{ |\eta_1|^2+|\eta_2|^2 }}\\
&=
 \f{ 1 }{\sqrt{ |\xi_1|^2+ |\xi_2|^2 } \, \sqrt{ |\eta_1|^2+|\eta_2|^2 }}
\det
\begin{pmatrix}
\xi_1&\xi_2\\
\eta_1&\eta_2
\end{pmatrix} \, .
\end{align*}
It is easy to verify that for $(\xi_1,\xi_2)\neq (0,0)$ we have
$$
\si_0(\xi_1,\xi_2,-\xi_1,-\xi_2) =\si_1(\xi_1,\xi_2,-\xi_1,-\xi_2)= 0.
$$

For higher order cancellation consider the examples
$$
\si_2(\xi_1,\xi_2,\eta_1,\eta_2)=
 \f{ \xi_1^2\eta_2^2 - 2\xi_1\xi_2\eta_1 \eta_2+\xi_2^2\eta_1^2}{
 ( |\xi_1|^2+ |\xi_2|^2+ |\eta_1|^2+|\eta_2|^2)^2 }
$$
and
$$
\si_3(\xi_1,\xi_2,\eta_1,\eta_2)=
 \f{ \xi_1^2\eta_2^2 - 2\xi_1\xi_2\eta_1 \eta_2+\xi_2^2\eta_1^2}{ ( |\xi_1|^2+ |\xi_2|^2)( |\eta_1|^2+|\eta_2|^2 ) }
$$
 both of which satisfy:
$$
\p_{\xi_1}^k\p_{\xi_2}^l \si_3(\xi_1,\xi_2,-\xi_1,-\xi_2) =
\p_{\xi_1}^k\p_{\xi_2}^l \si_4(\xi_1,\xi_2,-\xi_1,-\xi_2) =0 , \quad |\xi_1|^2+|\xi_2|^2\ne 0,
$$
for $(k,l) \in\{(0,1),(1,0),(0,0)\}$.
The symbols $\si_1$ and $\si_3$ are inspired by \cite{G}
and arise by expansions of the Hessian or by combinations of the Riesz transforms.
Examples of $\si_0$ and $\si_2$ are of Coifman-Meyer type (case (i) in the introduction)
while $\si_1$ and $\si_3$ are as in case (ii), i.e., sums of products of Calder\'on-Zygmund operators.

We generalize this example as follows:
\[
\sigma_{2N-2}(\xi_1,\xi_2,\eta_1,\eta_2)
=
\frac{1}{(|\xi_1|^2+ |\xi_2|^2+|\eta_1|^2+|\eta_2|^2)^{n_1+n_2+\cdots+n_N}}
\prod_{j=1}^N \det
\begin{pmatrix}
\xi_1{}^{n_j}&\xi_2{}^{n_j}\\
\eta_1{}^{n_j}&\eta_2{}^{n_j},
\end{pmatrix}
\]
where each $n_j$ is positive integer.
By the Leibniz rule we can check that
\[
\p_{\xi_1}^k\p_{\xi_2}^l \si_3(\xi_1,\xi_2,-\xi_1,-\xi_2) =
\p_{\xi_1}^k\p_{\xi_2}^l \si_4(\xi_1,\xi_2,-\xi_1,-\xi_2) =0 , \quad |\xi_1|^2+|\xi_2|^2\ne 0,
\]
as long as $k+l \le N-1$.

\medskip

Finally, we address the following question\footnote{posed by R. R. Coifman by personal communication}
and give a partial answer: 
Find a condition on a bilinear multiplier $B(f,g)$ such for any
two sequences $f_k\to f$ weakly and $g_k\to g$ weakly, then $B(f_k,g_k)\to B(f,g)$ weakly.
Suppose that $B$ is   given in multiplier form by
$$
B(f,g)(x) =
\int_{\mathbb R^n} \int_{\mathbb R^n}
\widehat{f}(\xi) \widehat{g}(\eta) \sigma(\xi,\eta) e^{2\pi i x\cdot (\xi+\eta)}
d\xi d\eta
$$
where $f,g$ are defined on $\mathbb R^n$ and $\sigma(\xi,\eta)$ is a Coifman-Meyer multiplier, i.e., it satisfies:
$$
\big|\partial_\xi^{\alpha} \partial_\eta^{\beta} \sigma(\xi,\eta) \big| \le C_{\alpha, \beta} (|\xi|+|\eta|)^{-|\alpha|-|\beta|}
$$
for sufficiently large multiindices $\alpha$, $\beta$. We provide a condition on $\sigma$ so that the
associated operator preserves weak convergence. Obviously the product $B(f,g) = fg$ does not preserve
weak convergence because the symbol $\si(\xi_1,\xi_2)=1$  fails to satisfy condition $\rm(v)$ below.

\begin{corollary}
Let $1<p<\infty$ and let $B$ be as above.
Suppose that $f_k$, $g_k$, $f$, $g$, $k=1,2,\dots$ are functions
on $\mathbb R^n$ that satisfy:
\begin{enumerate}
\item[{\rm(i)}] $\sup_{k} \| f_k \|_{L^p(\mathbb R^n)}\le C$.
\item[{\rm(ii)}] $\sup_{k} \| g_k \|_{L^{p'}(\mathbb R^n)}\le C$.
\item[{\rm(iii)}] $f_k\to f$ weakly in $L^p({\mathbb R}^n)$.
\item[{\rm(iv)}] $g_k\to g$ weakly in $L^{p'}({\mathbb R}^n)$.
\item[{\rm(v)}] $\si(\xi, -\xi) = 0 $ for all $\xi\neq 0$.
\item[{\rm(vi)}] $B(f_k,g_k) $ converges a.e. to $B(f,g)$.
\end{enumerate}
Then $B(f_k,g_k) $ converges to $B(f,g)$ weakly in $H^1({\mathbb R}^n)$  in the sense that
\begin{equation}\label{mn}
\int_{\mathbb R^n} \ B(f_k,g_k) \varphi \, dx
\to
\int_{\mathbb R^n} \ B(f ,g ) \varphi \, dx
\end{equation}
for all functions $\varphi \in {\rm VMO}(\mathbb R^n)$.
\end{corollary}

\begin{proof}
The boundedness of $B$ from $L^p({\mathbb R}^n)\times L^{p'}({\mathbb R}^n) $ 
to $H^1({\mathbb R}^n)$ can proved by combining condition $\rm(v)$ with Theorem~\ref{Thm-BiLin}
($N=1$) and the result in
\cite{GNNS17}; a version of this result was also proved by Dobyinski \cite[Lemme 3.8]{Do}; see also \cite{CLMS}.
It follows that
$$
\sup_{k} \| B(f_k,g_k) \|_{H^1} \le C_n \sup_{k} \| f_k \|_{L^p} \| g_k \|_{L^{p'}} \le C_n C^2 \, .
$$
Thus the sequence $ B(f_k,g_k),
k=1,2,\ldots$ is uniformly bounded in $H^1({\mathbb R}^n)$ and converges a.e. to $B(f,g)$.
Then we obtain \eqref{mn} as a consequence of the  result in \cite{JJ}.
\end{proof}

\end{document}